\documentclass[10pt,letterpaper]{amsart}
\usepackage[latin1]{inputenc}
\usepackage{amsmath}
\usepackage{amsfonts}
\usepackage{amssymb}
\usepackage{mathrsfs}
\usepackage{url}
\usepackage{color}
\usepackage{graphics}
\usepackage{graphicx}
\usepackage{booktabs}

\theoremstyle{definition}
\newtheorem{theorem}[equation]{Theorem}
\newtheorem{lemma}[equation]{Lemma}

\newtheorem{definition}[equation]{Definition}

\newtheorem{example}[equation]{Example}

\newtheorem{remark}[equation]{Remark}

\begin{document}

\title{Ordered Partitions and \\Drawings of Rooted Plane Trees}
\author{Qingchun Ren}
\date{Jan 27, 2014}

\begin{abstract}
We study the bounded regions in a generic slice of the hyperplane arrangement in $\mathbb{R}^n$ consisting of the hyperplanes defined by $x_i$ and $x_i+x_j$. The bounded regions are in bijection with several classes of combinatorial objects, including the ordered partitions of $[n]$ all of whose left-to-right minima occur at odd locations and the drawings of rooted plane trees with $n+1$ vertices. These are sequences of rooted plane trees such that each tree in a sequence can be obtained from the next one by removing a leaf.
\end{abstract}

\maketitle{}

\section{Introduction}\label{section-introduction}

We define the combinatorial objects to be studied in this paper. The first one is the following hyperplane arrangement on $\mathbb{R}^n$:
\begin{equation*}
\mathcal{H}_n = \{x_i,1\leq{}i\leq{}n\}\cup{}\{x_i+x_j,1\leq{}i<j\leq{}n\}.
\end{equation*}
Let $P$ be the affine hyperplane in $\mathbb{R}^n$ defined by
\begin{equation*}
P=\{l_1x_1+l_2x_2+\dotsb{}+l_nx_n=1\},
\end{equation*}
where $l_1\gg{}l_2\gg{}\dotsb{}\gg{}l_n>0$ (``$\gg{}$" means ``far greater than"). We are interested in the set of bounded regions of the hyperplane arrangement $\mathcal{H}_n\cap{}P = \{H\cap{}P:H\in{}\mathcal{H}_n\}$ in the affine space $P$.

\begin{definition}
An {\it ordered partition} of $[n]=\{1,2,\dotsc{},n\}$ (also called a {\it preferential arrangement} by Gross \cite{citation-gross}) is an ordered sequence $(A_1,\dotsc{},A_k)$ of disjoint non-empty subsets whose union is $[n]$. Each $A_i$ is called a {\it block}. A {\it left-to-right minimum} of $(A_1,\dotsc{},A_k)$ is $m_i=\mathrm{min}(A_1\cup{}\dotsb{}\cup{}A_i)$, where $1\leq{}i\leq{}k$. We say that a left-to-right minimum $m_i$ {\it occurs at an odd location} if $m_i\in{}A_j$ for some odd $j$.
\end{definition}

\begin{definition}
A {\it signed permutation} $((a_1,a_2,\dotsc{},a_n),\sigma{})$ of $[n]$ is a permutation $(a_1,a_2,\dotsc{},a_n)$ of $[n]$ together with a map $\sigma{}\colon{}[n]\to{}\{\pm{}1\}$. $\sigma{}(i)$ is called the {\it sign} of~$i$. It has {\it decreasing blocks} if $a_i>a_{i+1}$ for any two $a_i,a_{i+1}$ with the same sign. A {\it left-to-right minimum} of $((a_1,a_2,\dotsc{},a_n),\sigma{})$ is $m_i=\mathrm{min}(a_1,\dotsc{},a_i)$, where $1\leq{}i\leq{}n$. For simplicity, we indicate the sign of $a_i$ by writing $a_i^+$ or $a_i^-$.
\end{definition}

\begin{definition}
A {\it build-tree code} is a sequence $c_1c_2\dotsb{}c_n$ of pairs $c_i=(a_i,\sigma{}_i)$ where $0\leq{}a_i\leq{}i-1$ and $\sigma{}_i\in{}\{\pm{}1\}$ such that $(a_i,\sigma{}_i)\neq{}(0,-1)$. For simplicity, we write $a_i^+$ or $a_i^-$ instead of $(a_i,\sigma{}_i)$.
\end{definition}

\begin{definition}
An {\it increasing labeling} of a rooted plane tree $T$ of $n+1$ vertices, also called a {\it simple drawing} or a {\it heap order}, is a bijection $\lambda{}\colon{}T\to{}\{0,1,\dotsc{},n\}$ such that if $u,v\in{}T$ and $u$ is a child of $v$, then $\lambda{}(u)>\lambda{}(v)$. $\lambda{}(v)$ is called the {\it label} of $v$. An {\it increasingly labeled tree} is a rooted plane tree together with an increasing labeling. For simplicity, we identify a vertex with its label if there is no confusion.
\end{definition}

\begin{definition}
Let $(T,\lambda{})$ be an increasingly labeled tree. The {\it right associate} of a vertex $v$ is the sibling $u$ to the right of $v$ with the smallest label such that $\lambda{}(u)>\lambda{}(v)$, and $\lambda{}(u)$ is smaller than the labels of all siblings between $u$ and $v$, if such a $u$ exists. $(T,\lambda{})$ is a {\it Klazar tree} if it satisfies the following property: for any vertex $v$ with a right associate $u$, $v$ is not a leaf, and $\lambda{}(u)$ is larger than the minimum of all labels of children of $v$.
\end{definition}

Figure 1 shows two different linear extensions of the same tree. The tree in the right of Figure 1 is a Klazar tree. The tree in the left is not a Klazar tree, because the vertex $3$ has a right associate $4$, but $3$ is a leaf.

\begin{figure}[h]
\caption{Two increasingly labeled trees}
\centering
\includegraphics[width=0.4\textwidth]{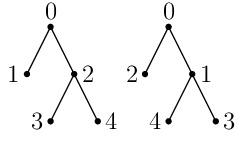}
\end{figure}

The last object is the set of {\it drawings} of rooted plane trees with $n+1$ vertices:

\begin{definition}
A {\it drawing} of a rooted plane tree $T$ is a sequence of rooted plane trees $T_0=\{\text{root}\},T_1,\dotsc{},T_n=T$ such that for each $0\leq{}i\leq{}n-1$, the tree $T_i$ can be obtained from $T_{i+1}$ by removing a leaf together with its pendant edge.
\end{definition}

The main result of this paper is

\begin{theorem}\label{theorem-main}
The following seven sets are in bijection:
\begin{itemize}
\item[(1)] The set of bounded regions in the affine hyperplane arrangement $\mathcal{H}_n\cap{}P$.
\item[(2)] The set of ordered partitions of $[n]$ all of whose left-to-right minima occur at odd locations.
\item[(3)] The set of signed permutations of $[n]$ with decreasing blocks all of whose left-to-right-minima have positive signs.
\item[(4)] The set of build-tree codes of length $n$ such that there is a $v^+$ after (but not necessarily adjacent to) each $v^-$.
\item[(5)] The set of build-tree codes of length $n$ such that there is a $v^+$ before (but not necessarily adjacent to) each $v^-$.
\item[(6)] The set of Klazar trees with $n+1$ vertices.
\item[(7)] The set of drawings of rooted plane trees with $n+1$ vertices.
\end{itemize}
Let $b_n$ be the common cardinality of these sets. Set $b_0=1$. Then, the sequence $\{b_n\}$ has the exponential generating function
\begin{equation*}
\sum_{n=0}^{\infty{}}b_n\frac{x^n}{n!}=\sqrt{\frac{e^x}{2-e^x}}.
\end{equation*}
\end{theorem}

The above generating function is due to Klazar \cite{citation-klazar} in a paper that discusses various counting problems of rooted plane trees. The bijections between (5), (6) and (7) are studied by Callan \cite{citation-callan}. Our notations are different from Callan's because we use the top-to-bottom convention for trees in this paper. Callan shows that $b_n$ also equals the number of perfect matchings on the set [2n] in which no even number is matched to a larger odd number. The sequence $\{b_n\}$ begins with
\begin{equation*}
1,1,2,7,35,226,1787,16717,\dotsc{}.
\end{equation*}
This sequence can be found in the On-Line Encyclopedia of Integer Sequences \cite[A014307]{citation-oeis}.

\begin{remark}
The number $b_n$ is related to the following {\em urn model}: one starts with $1$ black ball and $0$ white ball in an urn. In each step, one picks a ball randomly in the urn. If the ball is black, one puts that ball back to the urn together with another white ball. Otherwise, one puts that ball back to the urn together with two more black balls. Suppose that all balls are distinguishable. Then, $b_n$ equals the number of possible histories after $n$ steps. A detailed treatment on urn models can be found in \cite{citation-flajolet}.
\end{remark}

\begin{remark}
The number of bounded regions in $\mathcal{H}_n\cap{}P$ can be obtained by a simple application of the finite field method. However, it takes much more effort to establish a bijective proof.
\end{remark}

Our investigation originates from a {\it latent allocation model} in genomics \cite{citation-pachter}. Maximum likelihood estimation for this statistical model involves finding local maxima of the the function
\begin{equation*}
\prod_{i=1}^n|x_i|^{u_i}\prod_{1\leq{}i<j\leq{}n}|x_i+x_j|^{u_{ij}}
\end{equation*}
on the hyperplane $P'=\{x_1+x_2+\dotsb{}+x_n=1\}$, where $u_i,u_{ij}$ are generic positive integers. By a theorem of Varchenko \cite{citation-varchenko}, the ML degree of the statistical model, i.e. the number of local maxima of the above function equals the number of bounded regions in $\mathcal{H}_n\cap{}P'$. Our hyperplane $P$ can be considered as a deformation of $P'$:
\begin{equation*}
P=\{l_1x_1+l_2x_2+\dotsb{}+l_nx_n=1\},
\end{equation*}
where $l_1,\dotsc{},l_n$ are generic parameters. The number of bounded regions in $\mathcal{H}_n\cap{}P$ gives an upper bound on the number of bounded regions in $\mathcal{H}_n\cap{}P'$, and thus gives an upper bound on the maximum likelihood degree of the latent allocation model. Without loss of generality, we will assume that $l_1\gg{}l_2\gg{}\dotsb{}\gg{}l_n>0$.

Our hyperplane arrangement $\mathcal{H}_n$ is refined by the well-studied hyperplane arrangement of type $B_n$:
\begin{equation*}
\mathcal{B}_n = \{x_i,1\leq{}i\leq{}n\}\cup{}\{x_i\pm{}x_j,1\leq{}i<j\leq{}n\}.
\end{equation*}

Section \ref{section-bn} discusses the analogous problem for $\mathcal{B}_n$, and it shows that the bounded regions in $\mathcal{B}_n\cap{}P$ are in bijection with increasingly labeled trees with $n+1$ vertices. Based on this, Section \ref{section-hn} proves our main result, Theorem \ref{theorem-main}.

\section{Bounded regions in a slice of $\mathcal{B}_n$}\label{section-bn}

First, we consider the regions of the central hyperplane arrangement $\mathcal{B}_n$. The hyperplanes $x_i$ in $\mathcal{B}_n$ divide $\mathbb{R}^n$ into $2^n$ orthants. In each orthant, the hyperplanes $x_i\pm{}x_j$ divides the orthant into $n!$ regions, one for each total ordering of $|x_1|,\dotsc{},|x_n|$. Thus, for each signed permutation $((a_1,a_2,\dotsc{},a_n),\sigma{})$, we can associate it with a region $R$ of $\mathcal{B}_n$:
\begin{equation*}
R = \{(x_1,\dotsc{},x_n)\in{}\mathbb{R}^n:|x_{a_1}|>|x_{a_2}|>\dotsb{}>|x_{a_n}|,\mathrm{sgn}(x_i)=\sigma{}(i)\}.
\end{equation*}
Clearly, this is a bijection between regions of $\mathcal{B}_n$ and signed permutations of $[n]$.

\begin{lemma}\label{lemma-bn-regions}
Let $R$ be a region of $\mathcal{B}_n$. Let $((a_1,a_2,\dotsc{},a_n),\sigma{})$ be the corresponding signed permutation. Then
\begin{itemize}
\item[(a)] The polyhedron $R\cap{}P$ is nonempty and bounded if all left-to-right minima of $(a_1,a_2,\dotsc{},a_n)$ have positive signs.
\item[(b)] The polyhedron $R\cap{}P$ is empty if all left-to-right minima of $(a_1,a_2,\dotsc{},a_n)$ have negative signs.
\item[(c)] The polyhedron $R\cap{}P$ is nonempty and unbounded if neither of the above holds.
\end{itemize}
\end{lemma}

\begin{proof}
Let $e_1,e_2,\dotsc{},e_n$ be the unit vectors $(1,0,\dotsc{},0)$, $(0,1,\dotsc{},0)$, $\dotsc{}$, $(0,0,\dotsc{},1)$ in $\mathbb{R}^n$. Then, the region $R$ is the cone spanned over $\mathbb{R}_{\geq{}0}$ by the following $n$ vectors:
\begin{align*}
v_1 &= \sigma{}(a_1)e_{a_1},\\
v_2&=\sigma{}(a_1)e_{a_1}+\sigma{}(a_2)e_{a_2},\\
&\dotsb{},\\
v_n&=\sigma{}(a_1)e_{a_1}+\dotsb{}+\sigma{}(a_n)e_{a_n}.
\end{align*}
Let $L_i$ be the ray $\mathbb{R}_{\geq{}0}(v_i)$. Fix the linear form $f(x) = l_1x_1+\dotsb{}+l_nx_n$ on $\mathbb{R}^n$. Then,
\begin{equation*}
f(v_i) = \sigma{}(a_1)l_{a_1}+\dotsb{}+\sigma{}(a_i)l_{a_i}.
\end{equation*}
Since $l_1\gg{}\dotsb{}\gg{}l_n>0$, $f(v_i)$ has the same sign as $\sigma{}(\min{}(a_1,\dotsc{},a_i))$.

(a) By assumption, each left-to-right minima $\min{}(a_1,\dotsc{},a_i)$ has positive sign. Therefore, each $f(v_i)$ is positive. Hence, $P=\{f(x)=1\}$ intersects the ray $L_i$ at $v_i/f(v_i)$. Then, $R\cap{}P$ is the simplex with vertices $v_1/f(v_1),\dotsc{},v_n/f(v_n)$, which is nonempty and bounded.

(b) Similarly, each $f(v_i)$ is negative. Then, $f$ is negative on $R$. Thus, $P=\{f(x)=1\}$ does not intersect $R$.

(c) In this case, some $f(v_i)$ are positive and the others are negative. Say $f(v_i)>0$ and $f(v_j)<0$. Then, $P=\{f(x)=1\}$ intersects $L_i$ at $v_i/f(v_i)$. Moreover, $f(f(v_i)v_j-f(v_j)v_i))=0$. Hence, $R\cap{}P$ contains the affine ray $v_i/f(v_i)+\mathbb{R}^+(f(v_i)v_j-f(v_j)v_i))$. Thus, $R\cap{}P$ is nonempty and unbounded.
\end{proof}

\begin{theorem}\label{theorem-bn}
The following four sets are in bijection:
\begin{itemize}
\item[(1)] The set of bounded regions in $\mathcal{B}_n\cap{}P$.
\item[(2)] The set of signed permutations of $[n]$ all of whose left-to-right minima have positive signs.
\item[(3)] The set of build-tree codes of length $n$.
\item[(4)] The set of increasingly labeled trees with $n+1$ vertices.
\end{itemize}
Moreover, the common cardinality of these sets equals $(2n-1)!!=1\cdot{}3\cdot{}5\dotsb{}(2n-1)$.
\end{theorem}

\begin{proof}
$(1)\leftrightarrow(2)$. The regions of $\mathcal{B}_n\cap{}P$ are exactly $R\cap{}P$ for region $R$ of $\mathcal{B}_n$ such that $R\cap{}P$ is nonempty. Therefore, it follows from Lemma \ref{lemma-bn-regions} that the cardinalities of the sets (1) and (2) are equal.

$(2)\leftrightarrow(3)$. We construct a bijection between the set of build-tree codes of length $n$ and the set of signed permutations $[n]$ all of whose left-to-right minima have positive signs. Given a build-tree code $c_1c_2\dotsb{}c_n$, we construct a signed permutation. We start from the empty signed permutation. In each step, we look at $c_i$ and add one element to the signed permutation:
\begin{itemize}
\item[(i)] If $c_i=0^+$, then add $i$ to the beginning with positive sign.
\item[(ii)] If $c_i=j^+$ for $j>0$, then add $i$ immediately after $j$ with the opposite sign from $j$.
\item[(iii)] If $c_i=j^-$ for $j>0$, then add $i$ immediately after $j$ with the same sign as~$j$.
\end{itemize}
We obtain a signed permutation of $[n]$ in this way. In each step, if (ii) or (iii) hold, the left-to-right minima stay the same. If (i) holds, then $i$ becomes a new left-to-right minimum, and we construct it to have the positive sign. Thus, the signed permutation we constructed has the property that all of its left-to-right minima have positive signs. On the other hand, given a signed permutation of $[n]$ all of whose left-to-right minima have positive signs, we can reverse the construction and obtain a build-tree code. It is straightforward to verify that this gives a bijection.

$(3)\leftrightarrow{}(4)$. See Callan \cite{citation-callan}. We construct a bijection between the set of build-tree codes of length $n$ and the set of increasingly labeled trees with $n+1$ vertices. Given a build-tree code $c_1c_2\dotsb{}c_n$, we construct an increasingly labeled tree. We start from the rooted plane tree with no non-root vertices. In each step, we look at $c_i$ and add one leaf to the tree with label $i$:
\begin{itemize}
\item[(i)] If $c_i=0^+$, then add $i$ as the leftmost child of the root.
\item[(ii)] If $c_i=j^+$ for $j>0$, then add $i$ as the leftmost child of vertex $j$.
\item[(iii)] If $c_i=j^-$ for $j>0$, then add $i$ as the immediate right neighbor of $j$.
\end{itemize}
On the other hand, given an increasingly labeled tree with $n+1$ vertices, we can reverse the construction and obtain a build-tree code. It is straightforward to verify that this gives a bijection.

It is easy to see that there are $(2n-1)!!$ build-tree codes of length $n$, because each $c_i$ has exactly $2i-1$ independent choices.
\end{proof}

\begin{example}\label{example-signed-permutation}
Table 1 illustrates how we obtain a signed permutation of $\{1,2,3,4,5,6\}$ from the build-tree code $0^+1^-1^+1^+0^+3^+$ with the construction above.

\begin{table}[h!]
\caption{Constructing a signed permutation from a build-tree code}
\begin{tabular}{|c|c|c|c|}
\toprule{}
Step & Build-tree code & Rule applied & Signed permutation \\
\midrule{}
0 & & & \\
\midrule{}
1 & $0^+$ & (i) beginning, $+$ sign & $1^+$ \\
\midrule{}
2 & $0^+1^-$ & (iii) after $1$, same sign & $1^+2^+$ \\
\midrule{}
3 & $0^+1^-1^+$ & (ii) after $1$, opposite sign & $1^+3^-2^+$ \\
\midrule{}
4 & $0^+1^-1^+1^+$ & (ii) after $1$, opposite sign & $1^+4^-3^-2^+$ \\
\midrule{}
5 & $0^+1^-1^+1^+0^+$ & (i) beginning, $+$ sign & $5^+1^+4^-3^-2^+$ \\
\midrule{}
6 & $0^+1^-1^+1^+0^+3^+$ & (ii) after $3$, opposite sign & $5^+1^+4^-3^-6^+2^+$ \\
\cline{1-4}
\end{tabular}
\end{table}

\end{example}

\begin{example}\label{example-tree}
Table 2 illustrates how we obtain an increasingly labeled tree with $7$ vertices from the build-tree code $0^+1^+1^+1^-0^+3^+$ with the construction above.

\begin{table}[h!]
\caption{Constructing an increasingly labeled tree from a build-tree code}
\begin{tabular}{|c|c|c|c|}
\toprule{}
Step & Build-tree code & Rule applied & Tree \\
\midrule{}
\parbox[c]{1cm}{\centering{}0} & & & \parbox[c]{4cm}{\centering{}\includegraphics{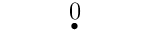}}\\
\midrule{}
\parbox[c]{1cm}{\centering{}1} & \parbox[c]{3cm}{\centering{}$0^+$} & \parbox[c]{2.5cm}{\centering{}(i) leftmost child of $0$} & \parbox[c]{4cm}{\centering{}\includegraphics{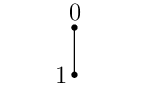}}\\
\midrule{}
\parbox[c]{1cm}{\centering{}2} & \parbox[c]{3cm}{\centering{}$0^+1^+$} & \parbox[c]{2.5cm}{\centering{}(ii) leftmost child of $1$} & \parbox[c]{4cm}{\centering{}\includegraphics{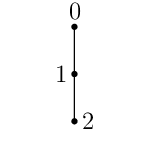}}\\
\midrule{}
\parbox[c]{1cm}{\centering{}3} & \parbox[c]{3cm}{\centering{}$0^+1^+1^+$} & \parbox[c]{2.5cm}{\centering{}(ii) leftmost child of $1$} & \parbox[c]{4cm}{\centering{}\includegraphics{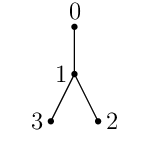}}\\
\midrule{}
\parbox[c]{1cm}{\centering{}4} & \parbox[c]{3cm}{\centering{}$0^+1^+1^+1^-$} & \parbox[c]{2.5cm}{\centering{}(iii) right neighbor of $1$} & \parbox[c]{4cm}{\centering{}\includegraphics{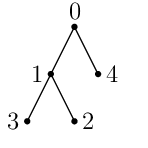}}\\
\midrule{}
\parbox[c]{1cm}{\centering{}5} & \parbox[c]{3cm}{\centering{}$0^+1^+1^+1^-0^+$} & \parbox[c]{2.5cm}{\centering{}(i) leftmost child of $0$} & \parbox[c]{4cm}{\centering{}\includegraphics{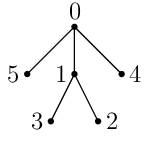}}\\
\midrule{}
\parbox[c]{1cm}{\centering{}6} & \parbox[c]{3cm}{\centering{}$0^+1^+1^+1^-0^+3^+$} & \parbox[c]{2.5cm}{\centering{}(ii) leftmost child of $3$} & \parbox[c]{4cm}{\centering{}\includegraphics{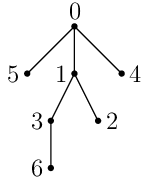}}\\
\cline{1-4}
\end{tabular}
\end{table}

\end{example}

\begin{remark}
Stanley \cite[Section 5.1]{citation-stanley2} computes the characteristic polynomial for the hyperplane arrangement $\mathcal{B}_n$. The number $(2n-1)!!$ is the signed constant term of the characteristic polynomial.
\end{remark}

\section{Bounded regions in a slice of $\mathcal{H}_n$}\label{section-hn}

First, we consider the regions of the central hyperplane arrangement $\mathcal{H}_n$. Let $(A_1,A_2,\dotsc{},A_k)$ be an ordered partition of $[n]$. We define a cone in $\mathbb{R}^n$ by
\begin{align*}
R^+ = \{(x_1,\dotsc{},x_n)\in{}\mathbb{R}^n\colon{}&x_i>0\text{ for }i\in{}A_j\text{ for odd }j,\\
&x_i<0\text{ for }i\in{}A_j\text{ for even }j,\\
&|x_{i_1}|>|x_{i_2}|\text{ for }i_1\in{}A_{j},i_2\in{}A_{j+1},1\leq{}j\leq{}k-1\}.
\end{align*}
Equivalently, the $3$rd condition above can be replaced by the condition that $x_{i_1}+x_{i_2}$ has the same sign as $(-1)^{j+1}$. Let $R^-=-R^+$.

\begin{lemma}\label{lemma-two-to-one}
There is a $2$ to $1$ map from the set of regions of $\mathcal{H}_n$ to the set of ordered partitions of $[n]$.
\end{lemma}

\begin{proof}
First, we notice that $R^+$ and $R^-$ are defined by inequalities involving only the linear forms in $\mathcal{H}_n$. Also, all signs of $x_i$ are implied by the defining inequalities of $R^+$ and $R^-$. These inequalities also imply the order of the $|x_i|$ except those $i$ in the same block. Since the $x_i$ with $i$ in the same block have the same sign, all signs of $x_i+x_j$ are implied by the defining inequalities. Therefore, $R^+$ and $R^-$ are indeed regions of $\mathcal{H}_n$.

On the other hand, given a generic point $(x_1,\dotsc{},x_n)\in{}\mathbb{R}^n$, we claim that it lies in a region of the form $R^+$ or $R^-$. Indeed, we order the $x_i$ by their absolute value: $|x_{p_1}|>|x_{p_2}|>\dotsb{}>|x_{p_n}|$. So, we get a permutation $(p_1,\dotsc{},p_n)$ of $[n]$. Then, we group together consecutive segments of the $p_i$ such that the $x_{p_i}$ has the same sign. In this way, we get an ordered partition of $[n]$. It follows that the point lies in $R^+$ or $R^-$, depending on the sign of $x_{p_1}$. Thus, we get a $2$ to $1$ correspondence from the set of regions of $\mathcal{H}_n$ to the set of ordered partitions of $[n]$.
\end{proof}

\begin{lemma}\label{lemma-hn-regions}
Let $(A_1,A_2,\dotsc{},A_k)$ be an ordered partition of $[n]$. Let $R^+$, $R^-$ be the two corresponding regions of $\mathcal{H}_n$. Then
\begin{itemize}
\item[(a)] If $(A_1,A_2,\dotsc{},A_k)$ has all left-to-right minima at odd locations, then $R^+\cap{}P$ is nonempty and bounded, and $R^-\cap{}P$ is empty.
\item[(b)] Otherwise, both $R^+\cap{}P$ and $R^-\cap{}P$ are nonempty and unbounded.
\end{itemize}
\end{lemma}

\begin{proof}
Since $\mathcal{B}_n$ refines $\mathcal{H}_n$, a region of $\mathcal{H}_n\cap{}P$ is nonempty (resp. unbounded) if and only if it contains a nonempty (resp. unbounded) region of $\mathcal{B}_n\cap{}P$. From the proof of Lemma \ref{lemma-two-to-one}, $R^+$ (resp. $R^-$) contains exactly the regions in $\mathcal{B}_n$ corresponding to signed permutations $((a_1,\dotsc{},a_n),\sigma{})$ such that $\sigma{}(a_1)=1$, (resp. $\sigma{}(a_1)=-1$) and $(A_1,A_2,\dotsc{},A_k)$ can be obtained from $(a_1,\dotsc{},a_n)$ by grouping together consecutive elements with the same signs. 

(a) Let $(x_1,\dotsc{},x_n)\in{}R^+$. Since the $x_i$ for $i\in{}A_j$ have sign $(-1)^{j-1}$, an odd location in $(A_1,A_2,\dotsc{},A_k)$ corresponds to elements in $((a_1,\dotsc{},a_n),\sigma{})$ with positive signs. Therefore, $(A_1,A_2,\dotsc{},A_k)$ has all left-to-right minima at odd locations if and only if all left-to-right minima of $((a_1,\dotsc{},a_n),\sigma{})$ have positive signs. Therefore, $R^+$ contains only regions of type (a) in Lemma \ref{lemma-bn-regions}, which are nonempty and bounded. Thus, $R^+\cap{}P$ is nonempty and bounded. Similarly, $R^-$ contains only regions of type (b) in Lemma \ref{lemma-bn-regions}. Thus, $R^-\cap{}P$ is empty.

(b) Similarly, both $R^+$ and $R^-$ contains only regions of type (c) in Lemma \ref{lemma-bn-regions}. Thus, both $R^+\cap{}P$ and $R^-\cap{}P$ are nonempty and unbounded.
\end{proof}

Now we prove our main result.

\begin{proof}[Proof of Theorem \ref{theorem-main}]
$(1)\leftrightarrow{}(2)$. The regions of $\mathcal{H}_n\cap{}P$ are exactly $R\cap{}P$ for regions $R$ of $\mathcal{H}_n$ such that $R\cap{}P$ is nonempty. Therefore, it follows from Lemma \ref{lemma-hn-regions} that the cardinalities of the sets (1) and (2) are equal.

$(2)\leftrightarrow{}(3)$. For each ordered partition $(A_1,A_2,\dotsc{},A_k)$, we construct a signed permutation with decreasing blocks by writing elements of each $A_i$ in decreasing order and concatenating them to form a permutation. The signs of the elements of $A_i$ is $(-1)^{i-1}$. For example, the ordered partition $(15,246,3)$ is sent to $5^+1^+6^-4^-2^-3^+$. It is clear that ordered partitions of $[n]$ all of whose left-to-right minima occur at odd locations are in bijection with signed permutations of $[n]$ with decreasing blocks all of whose left-to-right minima have positive signs.

$(3)\leftrightarrow{}(4)$.  A signed permutation fails to have decreasing blocks if and only if there are two adjacent elements $u,v$ with the same sign such that $u<v$. In other words, $v$ is added after $u$ with the same sign, and no more element is added after $u$ afterwards. Under the bijection described in the proof of Theorem \ref{theorem-bn}, this translates exactly to the condition that there is no $u^+$ after some $u^-$. Thus, the bijection sends signed permutations with decreasing blocks all of whose left-to-right minima have positive signs to build-tree codes such that there is a $v^+$ after each $v^-$, and vice versa.

$(4)\leftrightarrow{}(5)$. Given a build-tree code, we keep the numerals in the build-tree code, and reverse the order of the signs over each fixed numeral. For example, $0^+1^+1^-2^+2^+2^-$ goes to $0^+1^-1^+2^-2^+2^+$. In this way, the build-tree codes such that there is a $v^+$ before each $v^-$ are sent exactly to the build-tree codes such that there is a $v^+$ after each $v^-$, and vice versa.

$(5)\leftrightarrow{}(6)$. An increasingly labeled tree $(T,\lambda{})$ can be considered as a process of constructing the tree $T$ by adding vertices in the order determined by $\lambda{}$. The right associate of a vertex $v$, if it exists, is the first vertex added as the immediate right neighbor of $v$. Under the bijection described in the proof of Theorem \ref{theorem-bn}, the label of the right associate of $v$ corresponds to the location of the fist appearance of $v^-$ in the build-tree code. The condition that there is a $v^+$ before each $v^-$ translates to the condition that the right associate of $v$, if it exists, is added after at least one child of $v$. This is exactly the defining condition for Klazar trees. Thus, the bijection sends Klazar trees to build-tree codes such that there is a $v^+$ before each $v^-$, and vice versa.

$(6)\leftrightarrow{}(7)$. See Callan \cite{citation-callan}. We elaborate the proof for completion. Given a drawing $T_0,T_1,\dotsc{},T_n=T$, we can reconstruct an increasing labeling of $T$ as follows: suppose we have already constructed an increasing labeling of $T_{n-1}$. By definition, $T_{n-1}$ can be obtained from $T$ by removing a leaf. We label this leaf $n$, and label the rest of the tree in the same way as in $T_{n-1}$. In this way, we get an increasing labeling of $T$. To make the construction unambiguous, if there are multiple leaves in $T$ that can be removed to get $T_{n-1}$, we always choose the leftmost possible one. We claim that the resulting increasingly labeled tree $(T,\lambda{})$ is a Klazar tree. If it is not, then there is a vertex $v$ with a right associate $u$ such that either $v$ is a leaf or $\lambda{}(u)$ is smaller than the labels of all children of $v$. Since $T_{\lambda{}(u)}$ contains exactly the vertices in $T$ with label $\leq{}\lambda{}(u)$, the vertices $u$ and $v$ are adjacent leaves in $T_{\lambda{}(u)}$. Therefore, removing either $u$ or $v$ in $T_{\lambda{}(u)}$ results in the same rooted plane tree. Since $v$ is to the left of $u$, by the construction above, we would choose $v$ rather than $u$ in the $\lambda{}(u)$th step. So we get a contradiction. Thus, $(T,\lambda{})$ is a Klazar tree.

An increasingly labeled tree $(T,\lambda{})$ naturally gives a drawing $T_0,T_1,\dotsc{},T_n=T$, by setting $T_i$ to contain exactly the vertices with labels $\leq{}i$. Clearly this is a left inverse of the construction process above. It suffices to prove that different Klazar trees give different drawings. Assume that two different Klazar trees $(T,\lambda{})$ and $(T',\lambda{}')$ give the same drawing. Let $T_i$ (resp. $T'_i$) be the subtree of $T$ (resp. $T'$) spanned by vertices with labels $\leq{}i$. Then, $T_i$ and $T'_i$ are isomorphic. Thus, we can identify $T$ with $T'$. Let $k$ be the smallest positive integer such that $(T_k,\lambda{}|_{T_k})$ and $(T'_k,\lambda{}'|_{T'_k})$ are not isomorphic increasingly labeled trees. Moreover, both $T_k$ and $T'_k$ are Klazar trees. Without loss of generality, we may assume that $k=n$.

Let $u$ (resp. $u'$) be the vertex of $T$ labeled $n$ in $(T,\lambda{})$ (resp. $(T,\lambda{}')$). Note that $T_{n-1}$ (resp. $T'_{n-1}$) can be obtained from $T$ by removing $u$ (resp. $u'$). Then, both $u$ and $u'$ are leaves of $T$, and $u\neq{}u'$ by the minimality of $k$. Let $v$ be the lowest common ancestor of $u$ and $u'$. Let $v_1,v_2,\dotsc{},v_s$ be the children of $v$, ordered from left to right. Suppose that $u$ (resp. $u'$) is a descendent of $v_j$ (resp. $v_{j'}$). Then $j\neq{}j'$ by the choice of $v$. If neither $u$ nor $u'$ is a child of $v$, then the size of the subtree of $T_{n-1}$ rooted at $v_j$ would be $1$ smaller than the subtree of $T'_{n-1}$ rooted at $v_j$. If exactly one of $u$ or $u'$, say $u$, is a child of $v$, then $v$ would have one more child in $T'_{n-1}$ than in $T_{n-1}$. Both cases contradict our assumption that $T_{n-1}$ is isomorphic to $T'_{n-1}$. Thus, both $u$ and $u'$ are children of $v$. Without loss of generality, we may assume that $u'$ is to the right of $u$. Since $\lambda{}'(u')=n>\lambda{}'(u)$, the vertex $u$ has a right associate in $(T,\lambda{}')$. However, $u$ is a leaf. Thus, the condition for $(T,\lambda{}')$ being a Klazar tree is violated. So we get a contradiction.

It is shown by Klazar \cite{citation-klazar} that the cardinality of the set (7) has the given exponential generating function.
\end{proof}

We present an alternative proof by counting the cardinality of the set (2). We say that the {\it type} of an ordered partition $(A_1,\dotsc{},A_k)$ is the set $\{A_1,\dotsc{},A_k\}$, which is a partition of $[n]$. An ordered partition of type $\{ \{1\}, \{2\}, \dotsc{}, \{n\} \}$ is just a permutation of $[n]$.

\begin{lemma}
Let $p_n$ be the number of permutations of $[n]$ whose all left-to-right minima occurs at odd locations. Set $p_0=1$. Then
\begin{equation*}
\sum_{n=0}^{\infty{}}p_n\frac{x^n}{n!} = \sqrt{\frac{1+x}{1-x}}.
\end{equation*}
\end{lemma}

\begin{proof}
The proof is found in a post by Callan in \cite[A000246]{citation-oeis}. For any permutation $(a_1,\dotsc{},a_n)$ of $[n]$ all of whose left-to-right minima occur at odd locations, we can construct a permutation of $[n-1]$ by removing $a_n$ and decrementing all elements greater than $a_n$ by $1$. This new permutation has all left-to right minima at odd locations. On the other hand, for any permutation of $[n-1]$ whose all left-to-right minima occurs at odd locations and any $a_n\in{}[n]$, we can construct a permutation of $[n]$ by incrementing all elements greater than $a_n$ by $1$ and adding $a_n$ to the end. This new permutation has all left-to-right minima at odd locations if and only if $n$ is odd or $n$ is even and $a_n>1$. Therefore, from this correspondence we get $p_n=np_{n-1}$ for odd $n$ and $p_n=(n-1)p_{n-1}$ for even $n$. By induction, $p_{n}=((n-1)!!)^2$ for even $n$ and $p_n=n!!(n-2)!!$ for odd $n$.

Then,
\begin{align*}
\sum_{n=0}^{\infty{}}p_n\frac{x^n}{n!} &= \sum_{n\text{ even}}((n-1)!!)^2\frac{x^n}{n!}+\sum_{n\text{ odd}}n!!(n-2)!!\frac{x^n}{n!}\\
&=\sum_{m=0}^{\infty{}}((2m-1)!!)^2\frac{x^{2m}}{(2m)!}+\sum_{m=0}^{\infty{}}(2m+1)!!(2m-1)!!\frac{x^{2m+1}}{(2m+1)!}\\
&=\sum_{m=0}^{\infty{}}((2m-1)!!)^2\frac{x^{2m}}{(2m)!}+\sum_{m=0}^{\infty{}}((2m-1)!!)^2\frac{x^{2m+1}}{(2m)!}\\
&=(1+x)\sum_{m=0}^{\infty{}}((2m-1)!!)^2\frac{x^{2m}}{(2m)!}.
\end{align*}

On the other hand,
\begin{align*}
\frac{1}{\sqrt{1-x^2}} &= \sum_{m=0}^{\infty{}}{1/2 \choose m}(-1)^mx^{2m}\\
&=\sum_{m=0}^{\infty{}}\frac{(2m-1)!!}{2^m(m!)}x^{2m}\\
&=\sum_{m=0}^{\infty{}}((2m-1)!!)^2\frac{x^{2m}}{(2m)!}.
\end{align*}

So
\begin{equation*}
\sum_{n=0}^{\infty{}}p_n\frac{x^n}{n!} = \frac{1+x}{\sqrt{1-x^2}} = \sqrt{\frac{1+x}{1-x}}.
\end{equation*}
\end{proof}

\begin{lemma}
The number of ordered partitions of $[n]$ of type $\{A_1,\dotsc{},A_k\}$ all of whose left-to-right minima occurs at odd locations equals $p_k$.
\end{lemma}

\begin{proof}
We may replace each $A_i$ by $\{\min{}A_i\}$ without affecting the locations of the left-to-right minima. Therefore, we can reduce the problem to the case of ordered partitions of $k$ distinct numbers. Thus, the number is $p_k$.
\end{proof}

Let $b_n$ denote the number (2). Set $b_0=0$. Then, it follows from the composition formula \cite[Theorem 5.1.4]{citation-stanley} that
\begin{equation*}
\sum_{n=0}^{\infty{}}b_n\frac{x^n}{n!} = \sqrt{\frac{1+(e^x-1)}{1-(e^x-1)}} = \sqrt{\frac{e^x}{2-e^x}}.
\end{equation*}

\begin{example}
The build-tree codes $0^+1^+1^+1^-0^+3^+$ and $0^+1^-1^+1^+0^+3^+$ in Example \ref{example-signed-permutation} and Example \ref{example-tree} can be obtained from each other by reversing the sequence of signs over each fixed numeral in the build-tree code. Therefore, the objects in Table 1 and Table 2 are in bijection.
\end{example}

\begin{example}

Figure 2 shows the $7$ bounded regions in $\mathcal{H}_3\cap{}P$.

\begin{figure}[h!]
\caption{Seven bounded regions in $\mathcal{H}_3\cap{}P$}
\centering
\includegraphics[width=0.6\textwidth]{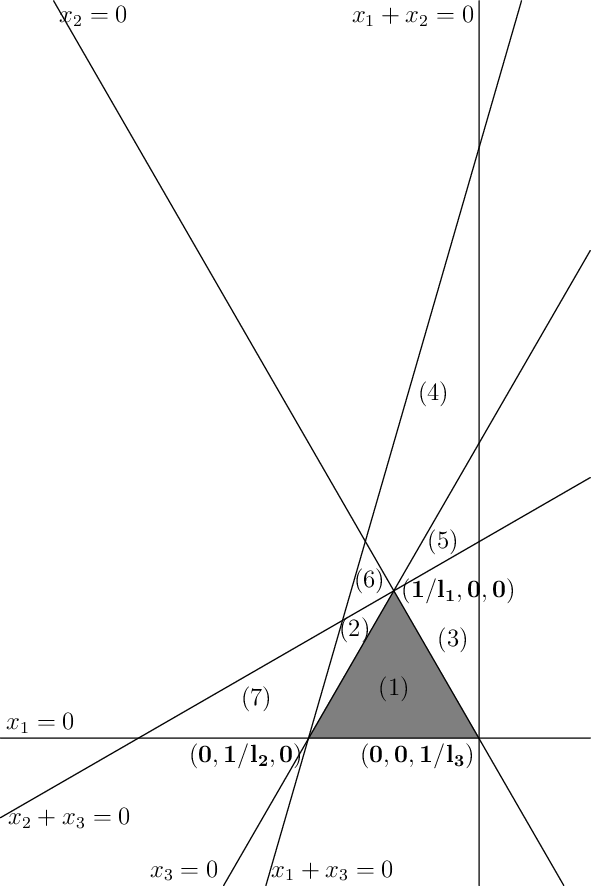}
\end{figure}

These $7$ bounded regions are labeled $(1), (2), \dotsc{}, (7)$. They are:
\begin{equation*}
\begin{array}{cl}
(1) & x_1>0,x_2>0,x_3>0 \\
(2) & x_1>0,x_2>0,x_3<0,|x_1|,|x_2|>|x_3| \\
(3) & x_1>0,x_2<0,x_3>0,|x_1|,|x_3|>|x_2| \\
(4) & x_1>0,x_2<0,x_3<0,|x_1|>|x_2|,|x_3| \\
(5) & x_1>0,x_2<0,x_3>0,|x_1|>|x_2|>|x_3| \\
(6) & x_1>0,x_2>0,x_3<0,|x_1|>|x_3|>|x_2| \\
(7) & x_1>0,x_2>0,x_3<0,|x_2|>|x_3|>|x_1| \\
\end{array}
\end{equation*}

Table 3 shows various objects that are in bijection with the $7$ bounded regions.

\begin{table}[h!]
\caption{The bijections for the $n=3$ case}
\begin{tabular}{|c|c|c|c|c|c|}
\toprule{}
\parbox[t]{1cm}{Label in Figure 2} & \parbox[t]{1.5cm}{Ordered partition} & \parbox[t]{1.5cm}{Signed permutation with decreasing blocks} & \parbox[t]{1.5cm}{Build-tree code such that there is a $v^+$ after each $v^-$} & \parbox[t]{1.5cm}{Build-tree code such that there is a $v^+$ before each $v^-$} & Klazar tree \\
\midrule{}
\parbox[c]{1cm}{\centering{}(1)} & \parbox[c]{1.5cm}{\centering{}$123$} & \parbox[c]{1.5cm}{\centering{}$3^+2^+1^+$} & \parbox[c]{1.5cm}{\centering{}$0^+0^+0^+$} & \parbox[c]{1.5cm}{\centering{}$0^+0^+0^+$} & \parbox[c]{3cm}{\centering{}\includegraphics[scale=0.85]{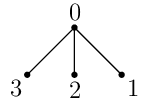}} \\
\midrule{}
\parbox[c]{1cm}{\centering{}(2)} & \parbox[c]{1.5cm}{\centering{}$12,3$} & \parbox[c]{1.5cm}{\centering{}$2^+1^+3^-$} & \parbox[c]{1.5cm}{\centering{}$0^+0^+1^+$} & \parbox[c]{1.5cm}{\centering{}$0^+0^+1^+$} & \parbox[c]{3cm}{\centering{}\includegraphics[scale=0.85]{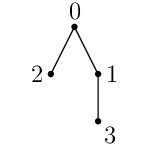}} \\
\midrule{}
\parbox[c]{1cm}{\centering{}(3)} & \parbox[c]{1.5cm}{\centering{}$13,2$} & \parbox[c]{1.5cm}{\centering{}$3^+1^+2^-$} & \parbox[c]{1.5cm}{\centering{}$0^+1^+0^+$} & \parbox[c]{1.5cm}{\centering{}$0^+1^+0^+$} & \parbox[c]{3cm}{\centering{}\includegraphics[scale=0.85]{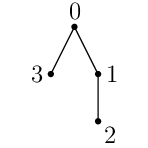}} \\
\midrule{}
\parbox[c]{1cm}{\centering{}(4)} & \parbox[c]{1.5cm}{\centering{}$1,23$} & \parbox[c]{1.5cm}{\centering{}$1^+3^-2^-$} & \parbox[c]{1.5cm}{\centering{}$0^+1^+1^+$} & \parbox[c]{1.5cm}{\centering{}$0^+1^+1^+$} & \parbox[c]{3cm}{\centering{}\includegraphics[scale=0.85]{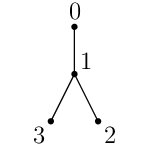}} \\
\midrule{}
\parbox[c]{1cm}{\centering{}(5)} & \parbox[c]{1.5cm}{\centering{}$1,2,3$} & \parbox[c]{1.5cm}{\centering{}$1^+2^-3^+$} & \parbox[c]{1.5cm}{\centering{}$0^+1^+2^+$} & \parbox[c]{1.5cm}{\centering{}$0^+1^+2^+$} & \parbox[c]{3cm}{\centering{}\includegraphics[scale=0.85]{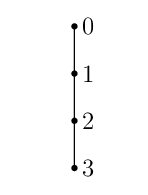}} \\
\midrule{}
\parbox[c]{1cm}{\centering{}(6)} & \parbox[c]{1.5cm}{\centering{}$1,3,2$} & \parbox[c]{1.5cm}{\centering{}$1^+3^-2^+$} & \parbox[c]{1.5cm}{\centering{}$0^+1^-1^+$} & \parbox[c]{1.5cm}{\centering{}$0^+1^+1^-$} & \parbox[c]{3cm}{\centering{}\includegraphics[scale=0.85]{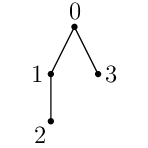}} \\
\midrule{}
\parbox[c]{1cm}{\centering{}(7)} & \parbox[c]{1.5cm}{\centering{}$2,3,1$} & \parbox[c]{1.5cm}{\centering{}$2^+3^-1^+$} & \parbox[c]{1.5cm}{\centering{}$0^+0^+2^+$} & \parbox[c]{1.5cm}{\centering{}$0^+0^+2^+$} & \parbox[c]{3cm}{\centering{}\includegraphics[scale=0.85]{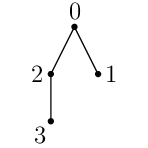}} \\
\cline{1-6}
\end{tabular}
\end{table}

\end{example}

\subsection*{Acknowledgements}

I would like to thank Bernd Sturmfels for guiding me through the entire project. Also, I would like to thank Lior Pachter for an insightful discussion that motivated the project. This project was supported by DARPA (grant DARPA-11-65-Open-BAA-FP-068).

\noindent
\footnotesize {\bf Author's address}:
Department of Mathematics, University of California, Berkeley, CA 94720, USA.
{\tt qingchun@berkeley.edu}

\end{document}